\documentclass{amsart}
\usepackage{mathptmx}
\usepackage{amssymb}
\usepackage{amsfonts}
\usepackage{amsmath}
\usepackage{graphicx}
\usepackage{shadow}
\usepackage[all]{xy}
\newtheorem{thm}{Theorem}[section]
\newtheorem{corollary}[thm]{Corollary}
\newtheorem{lemma}[thm]{Lemma}

\newtheorem{conjecture}[thm]{Conjecture}

\theoremstyle{definition}
\newtheorem{definition}[thm]{Definition}
\newtheorem{example}[thm]{Example}

\theoremstyle{remark}
\newtheorem{remark}[thm]{Remark}

\DeclareMathOperator{\qf}{qf}

\DeclareMathOperator{\Ker}{Ker}

\DeclareMathOperator{\Spec}{Spec}
\DeclareMathOperator{\Max}{Max}

\DeclareMathOperator{\w.dim}{w.dim}
\DeclareMathOperator{\fd}{fd}
\newcommand{\field}[1]{\mathbb{#1}}

\newcommand{\R}{\field{R}}

\newcommand{\Z}{\field{Z}}

\def\1{{\rm (1)}}
\def\2{{\rm (2)}}
\def\3{{\rm (3)}}
\def\4{{\rm (4)}}
\def\5{{\rm (5)}}

\begin{document}

\title[Trivial extensions defined by Pr\"ufer conditions]{Trivial extensions defined by Pr\"ufer conditions}

\author{C. Bakkari}
\address{Department of Mathematics, Faculty of Sciences and Technology, P. O. Box 2202,
University S. M. Ben Abdellah, Fez 30000, Morocco} \email{cbakkari@hotmail.com}

\author{S. Kabbaj}
\address{Department of Mathematics, King Fahd University of Petroleum \& Minerals, P. O. Box 5046, Dhahran 31261, KSA}
\email{kabbaj@kfupm.edu.sa}
\thanks{The second author was supported by KFUPM under Research Project \# MS/RING/368. The third author would like to thank KFUPM for its hospitality}

\author{N. Mahdou}
\address{Department of Mathematics, Faculty of Sciences and Technology, P. O. Box 2202,
University S. M. Ben Abdellah, Fez 30000, Morocco} \email{mahdou@hotmail.com}

\date{\today}

\begin{abstract}
This paper deals with well-known extensions of the Pr\"ufer domain concept to arbitrary commutative rings. We investigate the transfer of these notions in trivial ring extensions (also called idealizations) of commutative rings by modules and then generate original families of rings with zerodivisors subject to various Pr\"ufer conditions. The new examples give further evidence for the validity of Bazzoni-Glaz conjecture on the weak dimension of Gaussian rings. Moreover, trivial ring extensions allow us to widen the scope of validity of Kaplansky-Tsang conjecture on the content ideal of Gaussian polynomials.
\end{abstract}
\maketitle

\section{Introduction}

All rings considered in this paper are commutative with identity elements and all modules are
unital. In 1932, Pr\"ufer introduced and studied integral domains in which every non-zero finitely generated ideal is
invertible \cite{P}. In 1936, Krull \cite{K} named these rings after H. Pr\"ufer and stated
equivalent conditions that make a domain Pr\"ufer. Since then, ``Pr\"ufer domains have assumed a central role in the development of multiplicative
ideal theory through numeral equivalent forms. These touched on many areas of commutative algebra, e.g., valuation theory, arithmetic relations
on the set of ideals, $*$-operations, and polynomial rings; in addition to several homological characterizations" (Gilmer \cite{Gi}).

The extension of this concept to rings with zerodivisors gave rise to five classes of Pr\"ufer-like rings featuring some homological apsects (Bazzoni-Glaz \cite{BG} and Glaz \cite{G3}); namely: \1 Semihereditary ring, i.e., every finitely generated ideal is projective (Cartan-Eilenberg \cite{CE}); \2 Ring with weak dimension at most one (Glaz \cite{G1,G2}); \3 Arithmetical ring, i.e., every finitely generated ideal is locally principal (Fuchs \cite{Fu}); \4 Gaussian ring, i.e., $c(fg)=c(f)c(g)$ for any polynomials $f,g$ with coefficients in the ring, where $c(f)$ denotes the content of $f$ (Tsang \cite{T}); \5 Pr\"ufer ring, i.e., every finitely generated regular ideal is invertible or, equivalently, projective (Butts-Smith \cite{BS} and Griffin \cite{Gr}).

While, in the domain context, all these forms coincide with the definition of a Pr\"ufer domain, Glaz \cite{G3} provided
examples which show that all these notions are distinct in the context of arbitrary rings. The following diagram of implications
 summarizes the relations between them \cite{BG,BG2,G2,G3,LR,Lu,T}:

\begin{center}
Semihereditary $\Rightarrow$ weak dimension $\leq 1$ $\Rightarrow$ Arithmetical $\Rightarrow$ Gaussian $\Rightarrow$ Pr\"ufer
\end{center}

 It is notable that original examples -for each one of the above classes- are rare in the literature. This paper investigates the transfer of the above-mentioned Pr\"ufer conditions to trivial ring extensions. Our results generate new examples which enrich the current literature with new families of Pr\"ufer-like rings with zerodivisors. Particularly, we obtain further evidence for the validity of Bazzoni-Glaz conjecture sustaining that ``the weak dimension of a Gaussian ring is 0, 1, or $\infty$" \cite{BG2}. Moreover, trivial ring extensions offer the possibility to widen the scope of validity of the content conjecture of Kaplansky and Tsang which was extended to investigate rings where ``every Gaussian polynomial has locally principal content ideal" \cite{AK,BG,GV,HH,LR,Lu,T}. Notice that both conjectures share the common context of rings with zerodivisors. This very fact lies behind our motivation for studying the Gaussian condition and related concepts in trivial ring extensions.

Let $A$ be a ring and $E$ an $A$-module. The trivial ring extension of $A$ by $E$ (also called the idealization of $E$ over $A$) is the ring $R:= A\propto~E$ whose underlying group is $A \times E$ with multiplication given by $(a,e)(a',e') = (aa', ae'+a'e)$. For the reader's convenience, recall that if $I$ is an ideal of $A$ and $E'$ is a submodule of $E$ such that $IE \subseteq E'$, then $J := I \propto E'$ is an ideal of $R$; ideals of $R$ need not be of this form \cite[Example 2.5]{KM}. However, prime (resp., maximal) ideals of $R$ have the form  $p\propto E$, where $p$ is a prime (resp., maximal) ideal of $A$ \cite[Theorem 25.1(3)]{H}. Suitable background on commutative trivial ring extensions is \cite{G1,H}.

Section~\ref{sec:2} deals with trivial ring extensions of the form $R:=A \propto~B$, where $A \subseteq B$ is an extension of integral domains. Precisely, we examine the possible transfer of Pr\"ufer ring  conditions to $R$. The main result asserts that ``$R$ is Gaussian (resp., Arithmetical) if and only if $A$ is Pr\"ufer with $K\subseteq B$ (resp., $K=B$)." This generates new examples of non-arithmetical Gaussian rings as well as arithmetical rings with weak dimension strictly greater than one. Recall that classical examples of non-semihereditary arithmetical rings stem from Jensen's 1966 result \cite{J} as non-reduced principal rings, e.g., $\Z/n^{2}\Z$ for any integer $n\geq 2$. In this line, we provide a new family of examples of non-finite conductor arithmetical rings, hence quite far from being principal. We also establish a result on the weak dimension of these constructions which happens to corroborate Bazzoni-Glaz conjecture (cited above).

In their recent paper devoted to Gaussian properties, Bazzoni and Glaz have proved that a Pr\"ufer ring satisfies any of the other four Pr\"ufer
conditions if and only if its total ring of quotients satisfies that same condition \cite[Theorems 3.3 \& 3.6 \& 3.7 \& 3.12]{BG2}.
This fact narrows the scope of study to the class of total rings of quotients. Section~\ref{sec:3} investigates Pr\"ufer conditions in a special class of total rings of quotients; namely, those arising as trivial ring extensions of local rings by vector spaces over the residue fields. The main result establishes that ``if $(A,M)$ is a non-trivial local ring and $E$ a nonzero $\frac{A}{M}$-vector space, then $R:=A\propto~E$ is a non-arithmetical total ring of quotients. Moreover, $R$ is a Gaussian ring if and only if so is $A$." This enables us to build new examples of non-arithmetical Gaussian total rings of quotients or non-Gaussian total rings of quotients (which are necessarily Pr\"ufer). Further the weak dimension of these constructions turns out to be infinite which subjects, here too, the new examples of Gaussian rings to Bazzoni-Glaz conjecture.

A problem initially associated with Kaplansky and his student Tsang \cite{AK,BG,GV,Lu,T} and
also termed as Tsang-Glaz-Vasconcelos conjecture in \cite{HH} sustained that ``every nonzero Gaussian polynomial over a domain has an invertible (or,
equivalently, locally principal) content ideal." It is well-known that a polynomial over any ring is Gaussian if its content ideal is locally
principal. The converse is precisely the object of Kaplansky-Tsang-Glaz-Vasconcelos conjecture extended to those rings where ``every Gaussian
polynomial has locally principal content ideal. The objective of Section~\ref{sec:4} is to validate this conjecture in a large family of rings distinct from the three classes of arithmetical rings, of locally domains, and of locally approximately Gorenstein rings, where the conjecture holds so far. This was made possible by the main result which states that a trivial ring extension of a domain by its quotient field satisfies the condition that ``every Gaussian polynomial has locally principal content ideal." We end up with a conjecture that equates the latter condition with the local irreducibility of the zero ideal. This would offer an optimal solution to the Kaplansky-Tsang-Glaz-Vasconcelos conjecture that recovers all previous results. The section closes with a discussion -backed with examples- which attempts to rationalize this statement.

\section{Extensions of domains}\label{sec:2}

This section explores trivial ring extensions of the form $R:=A \propto~B$, where $A \subseteq B$ is an extension of integral domains. Notice in this context that $(a,b)\in R$ is regular if and only if $a\not=0$. The main result (Theorem~\ref{sec:2.1}) examines the transfer of Pr\"ufer conditions to $R$ and hence generates new  examples of non-arithmetical Gaussian rings and of arithmetical rings with weak dimension $\gvertneqq1$.

In 1969, Osofsky proved that the weak dimension of an arithmetical ring is either $\leq1$ or infinite \cite{Os}. In 2005, Glaz proved Osofsky's result in the class of coherent Gaussian rings \cite[Theorem 3.3]{G2}. Recently, Bazzoni and Glaz conjectured that ``the weak dimension of a Gaussian ring is 0, 1, or $\infty$" \cite{BG2}. Theorem~\ref{sec:2.1} validates this conjecture for the class of all Gaussian rings emanating from these constructions. Moreover, Example~\ref{sec:2.7} widens its scope of validity beyond coherent Gaussian rings.

\begin{thm}\label{sec:2.1}
Let $A\subseteq B$ be an extension of domains and $K :=\qf(A)$. Let $R:=A \propto~B$ be the trivial ring extension of $A$ by $B$. Then:\\
\1 $R$ is Gaussian if and only if $R$ is  Pr\"ufer if and only if $A$ is Pr\"ufer with $K\subseteq B$.\\
\2 $R$ is arithmetical if and only if $A$ is Pr\"ufer with $K=B$.\\
\3 $\w.dim(R)=\infty$.
\end{thm}

The proof of the theorem involves the following lemmas of independent interest.

\begin{lemma}\label{sec:2.2}
Let $A$ be a ring, $E$ an nonzero $A$-module, and $R:=A \propto~E$. If $R$ is Gaussian (resp., arithmetical), then so is $A$.
\end{lemma}

\begin{proof}
Straightforward since the arithmetical and Gaussian properties are stable under factor rings (here $A\cong \frac{R}{0\propto E}$).
\end{proof}

Notice that Lemma~\ref{sec:2.2} does not hold for the Pr\"ufer property as shown by Example~\ref{sec:2.6}.

\begin{lemma}\label{sec:2.2.1}
Let $K$ be a field, $E$ a nonzero $K$-vector space, and $R:=K \propto~E$. Then $\w.dim(R)=\infty$.
\end{lemma}

\begin{proof}
Let $\{f_{i}\}_{i \in I}$ be a basis of the $K$-vector space $E$ and let $J:=0\propto E$. Consider the $R$-map $R^{(I)}\buildrel u \over\rightarrow J$ defined by $u((a_{i},e_{i})_{i \in I}) =(0,\sum_{i \in I}a_{i}f_{i})$. Clearly, $\Ker(u) =0 \propto E^{(I)}$. Here we are identifying $R^{(I)}$ with $A^{(I)}\propto E^{(I)}$ as $R$-modules.
We have the exact sequence of $R$-modules:
$$0 \rightarrow 0 \propto E^{(I)} \rightarrow R^{(I)}  \buildrel u \over\rightarrow  J \rightarrow 0.$$
We claim that $J$ is not flat. Otherwise, by \cite[Theorem 3.55]{Ro}, we obtain
$$0\propto E^{(I)}=J^{(I)}=JR^{(I)}=(0 \propto E^{(I)})\cap JR^{(I)}=(0 \propto E^{(I)})J=0,$$
a contradiction. Therefore the above exact sequence  yields
$$\fd(J)= \fd( J^{(L)}) \leq \fd(J)-1.$$
This forces the flat dimension of $J$ and hence the weak dimension of $R$ to be infinite.
\end{proof}

\begin{proof}[Proof of  Theorem~\ref{sec:2.1}]
(1) We need only prove the following implications: \begin{center}$R$ Pr\"ufer $\Rightarrow$ $A$ Pr\"ufer with $K\subseteq B$ $\Rightarrow$ $R$
Gaussian\end{center}

Assume $R$ is a Pr\"ufer ring. We wish to show first that $K\subseteq B$ in the case when $A$ is local. Let $x\not=0\in A$ and let
$I:=((x,0),(x,1))R$, a finitely generated regular ideal of $R$. Then $I$ is invertible and hence principal (since $R$ is local too). Write $I=(a,b)R$
for some $a\in A$ and $b\in B$. Clearly, $a=ux$ for some invertible element $u$ in $A$, hence $I=(ux,b)R=(x,u^{-1}b)R$. Further $(x,0)\in I$ yields
$u^{-1}b=b'x$ for some $b'\in B$. It follows that $I=(x,b'x)R=(x,0)(1,b')R=(x,0)R$ since $(1,b')$ is invertible. But $(x,1)\in I$ yields $1=xb''$ for
some $b''\in B$. Therefore $K\subseteq B$. Next suppose $A$ is not necessarily local and let $q\in\Spec(B)$ and $p:=q\cap A$. Clearly,
$S:=(A\setminus p)\times 0$ is a multiplicatively closed subset of $R$ with the feature that $\frac{r}{1}$ is regular in $S^{-1}R$ if and only if $r$
is regular in $R$. So finitely generated regular ideals of $S^{-1}R$ originate from finitely generated regular ideals of $R$. Hence $A_{p}\propto
B_{p}=S^{-1}R$ is a Pr\"ufer ring. Whence $K=\qf(A_{p})\subseteq B_{p}\subseteq B_{q}$. It follows that $K\subseteq B=\bigcap B_{q}$, where $q$
ranges over $\Spec(B)$, as desired. Now, one can easily check that $K\subseteq B$ implies $K\propto B=Q(R)$, the total ring of quotients of $R$.
Moreover, let $f=\sum (k_{i},b_{i})x^{i}$ and $g=\sum (k'_{j},b'_{j})x^{j}$ be two polynomials in $Q(R)[x]$. If there is $i$ or $j$ such that
$k_{i}\not=0$ or $k'_{j}\not=0$, then $(k_{i},b_{i})$ or $(k'_{j},b'_{j})$ is invertible, hence $c(f)=Q(R)$ or $c(g)=Q(R)$, whence $c(fg)=c(f)c(g)$
(this is Gauss lemma which asserts that a polynomial with unit content is Gaussian). If $k_{i}=k'_{j}=0$ for all $i$ and $j$, then
$c(fg)=0=c(f)c(g)$. Consequently, $Q(R)$ is a Gaussian ring and so is $R$ by \cite[Theorem 3.3]{BG2}. By Lemma~\ref{sec:2.2}, $A$ is a Pr\"ufer
domain, completing the proof of the first implication.

Assume $A$ is a Pr\"ufer domain with $K\subseteq B$. Let $I$ be a nonzero finitely generated ideal of $R$ minimally generated by $(a_{1},b_{1}),
\ldots, (a_{n},b_{n})$. For any $a\not=0\in A$ and any $b,b'\in B$, we have $(0,b')=(a,b)(0,a^{-1}b')$. So minimality forces either $a_{i}=0$ for
each $i$ or $a_{i}\not=0$ for each $i$. In the first case, $I^{2}=0$ and hence $I$ is not a regular ideal. Next assume $a_{i}\not=0$ for each $i$. It
follows that $I=(\sum Aa_{i})\propto B$ since $(a_{i}, b)=(a_{i},b_{i})(1,a_{i}^{-1}(b-b_{i}))$ for each $i$ and any $b\in B$. Since $A$ is a
Pr\"ufer domain, $J:=\sum Aa_{i}$ is invertible and $aJ^{-1}$ is an ideal of $A$ for some $a\not= 0\in A$.
\(\begin{array}{lcl}
\hbox{So } (a,0)^{-1}(aJ^{-1} \propto B)I   & =     &(a,0)^{-1}(aJ^{-1} \propto B)(J \propto B)\\
                                    & =     &(a,0)^{-1}(aJ^{-1}J \propto B)\\
                                    & =     &(a,0)^{-1}(aA\propto B)\\
                                    & =     &R.
\end{array}\)\\
Consequently, $R$ is a Pr\"ufer ring and hence Gaussian by \cite[Theorem 3.3]{BG2}, completing the proof of (1).

(2) Assume $R$ is an arithmetical ring. By (1), $A$ is Pr\"ufer with $K\subseteq B$. So $K\propto B=Q(R)$ is arithmetical since it is a localization
of $R$. Let $b\not=0\in B$. Then $I:=((0,1),(0,b))Q(R)$ is principal and hence $I:=(0,b')Q(R)$ for some $b'\not=0\in B$. Further $(0,b)\in I$ yields
$b=kb'$ for some some $k\not=0\in K$, and then $I:=(0,k^{-1}b)Q(R)$. Moreover $(0,1)\in I$ yields $1=k'k^{-1}b$ for some $k'\not=0\in K$. It follows
that $b\in K$ and thus $K=B$. Conversely, assume $A$ is Pr\"ufer with $K=B$. By (1), $R=A\propto K$ is Gaussian. Moreover $Q(R)=K \propto K$ is a
principal ring (a fortiori arithmetical) since it has a unique nonzero proper ideal $M :=0 \propto K =T(0,1)$. By \cite[Theorem 3.5]{BG2}, $R$ is
arithmetical, completing the proof of (2).

(3) Let $S :=A\setminus\{0\}$. So $T:=S\times 0$ is a multiplicatively closed subset of $R$. By Lemma~\ref{sec:2.2.1}, $\w.dim(T^{-1}R)=\w.dim(K \propto KB)=\infty$. So  $\w.dim(R)=\infty$. This completes the proof of the theorem.
\end{proof}

The following corollary is an immediate consequence of Theorem~\ref{sec:2.1}.

\begin{corollary}\label{sec:2.3}
Let $D$ be a domain, $K :=\qf(D)$, and $R:=D \propto~K$. Then the following statements are equivalent:\\
\1 D is a Pr\"ufer domain;\\
\2 $R$ is an arithmetical ring;\\
\3 $R$ is a Gaussian ring;\\
\4 $R$ is a Pr\"ufer ring.
\qed\end{corollary}

Recall Jensen's 1966 result: ``for a ring $R$, $\w.dim(R)\leq 1$ if and only if $R$ is arithmetical and reduced" \cite{J}. Classical examples of
arithmetical rings with weak dimension $\gneqq 1$ stem from Jensen's result as non-reduced principal rings, e.g., $\Z/n^{2}\Z$ for any integer $n\geq
2$. In this vein, Theorem~\ref{sec:2.1} generates a new family of examples quite far from being principal as shown below. For this purpose, recall
that a ring $R$ is finite conductor if $aR\cap bR$ and $(0:c)$ are finitely generated for any $a,b,c\in R$ \cite{G12}. The class of finite conductor rings properly contains the class of coherent (a fortiori, Noetherian and hence principal) rings \cite{G12,KM}.

\begin{example}\label{sec:2.4}
Let $D$ be any Pr\"ufer domain which is not a field and $K :=\qf(D)$. Then $R:=D \propto~K$ is an arithmetical ring with $\w.dim(R)=\infty$.
Moreover, $R$ is not a finite conductor ring by \cite[Thoerem 2.8]{KM} and hence not coherent.
\end{example}

Also, Theorem~\ref{sec:2.1} enriches the literature with new examples of non-arithmetical Gaussian rings, as shown below.

\begin{example}\label{sec:2.5}
Let $K\subsetneqq L$ be a field extension. Then $R:=K \propto~L$ is a Gaussian ring which is not arithmetical.
\end{example}

The next example shows that Theorem~\ref{sec:2.1} widens the scope of validity of Bazzoni-Glaz conjecture beyond the class of coherent Gaussian rings.

\begin{example}\label{sec:2.7}
Let $\Z$ and $\R$ denote the ring of integers and field of real numbers, respectively. Then $R:=\Z_{(2)} \propto~\R$ satisfies the following statements:\\
\1 $R$ is a Gaussian ring,\\
\2 $R$ is not an arithmetical ring,\\
\3 $R$ is not a coherent ring,\\
\4 $\w.dim(R)=\infty$.
\end{example}

\begin{proof}
Assertions \1, \2, and \4 hold by direct application of Theorem~\ref{sec:2.1}. It remains to prove \3. Indeed, consider the following exact sequence over $R$
$$0 \rightarrow 0 \propto~\R \rightarrow R  \buildrel u \over\rightarrow  R(0,1)=0 \propto~\Z_{(2)} \rightarrow 0$$
where $u$ is defined by $u(a,b)=(a,b)(0,1)=(0,a)$. Now $0 \propto~\R$ is not finitely generated as an $R$-module (otherwise $\R$ would be finitely generated as a $\Z_{(2)}$-module). Hence $0 \propto~\Z_{(2)}$ is a finitely generated ideal of $R$ that is not finitely presented. Whence $R$ is not coherent, as desired.
\end{proof}

The next example illustrates the failure of Theorem~\ref{sec:2.1}, in general, beyond the context of domain extensions.

\begin{example}\label{sec:2.6}\rm
Let $(A,M)$ be a non-valuation local domain, $E$ a nonzero $A$-module with $ME=0$, and $B:=A \propto~E$. Then $R:=A \propto~B$ is a Pr\"ufer ring
which is not Gaussian.
\end{example}

\begin{proof}
Indeed, one can easily check that $R$ is a total ring of quotients and hence a Pr\"ufer ring. By Lemma~\ref{sec:2.2}, $R$ is
not Gaussian.
\end{proof}

\section{A class of total rings of quotients}\label{sec:3}

In a recent paper devoted to Gaussian properties, Bazzoni and Glaz have proved that a Pr\"ufer ring satisfies any of the other four Pr\"ufer
conditions (mentioned above) if and only if its total ring of quotients satisfies that same condition \cite[Theorems 3.3 \& 3.6 \& 3.7 \& 3.12]{BG2}.
This fact narrows the scope of study to the class of total rings of quotients.

This section investigates Pr\"ufer conditions in a particular class of total rings of quotients; namely, those arising as trivial ring extensions of
local rings by vector spaces over the residue fields. The main result (Theorem~\ref{sec:3.1}) enriches the literature with original examples of
non-arithmetical Gaussian total rings of quotients as well as non-Gaussian total rings of quotients (which are necessarily Pr\"ufer).  The theorem validates Bazzoni-Glaz conjecture for the class of Gaussian rings emanating from these constructions.

\begin{thm}\label{sec:3.1}
Let $(A,M)$ be a local ring and $E$ a nonzero $\frac{A}{M}$-vector space. Let $R:=A\propto~E$ be the trivial ring extension of $A$ by $E$. Then:\\
\1 $R$ is a total ring of quotients and hence a Pr\"ufer ring.\\
\2 $R$ is Gaussian if and only if $A$ is Gaussian.\\
\3 $R$ is arithmetical if and only if $A:=K$ is a field and $\dim_{K}E=1$.\\
\4 $\w.dim(R)\gvertneqq 1$. If $M$ admits a minimal generating set, then $\w.dim(R)=\infty$.
\end{thm}

\begin{proof}
\1 Straightforward.

\2 By Lemma~\ref{sec:2.2}, only the sufficiency has to be proved. Assume $A$ is a Gaussian ring and let $F =\sum (a_{i},e_{i})x^{i}$ be a polynomial in $R[x]$. If $a_{i} \notin M$
for some $i$, then $(a_{i},e_{i})$ is invertible in $R$, hence $F$ is Gaussian. Now assume $a_{i} \in M$ for each $i$ and let $G =\sum
(a'_{j},e'_{j})x^{j} \in R[x]$. We may suppose, without loss of generality, that $a'_{j}\in M$ for each $j$. Let $f =\sum a_{i}x^{i}$ and $g =\sum
a'_{j}x^{j}$ in $A[x]$. One can easily check that $ME=0$ yields the following \\
\(\begin{array}{lcl}
c(FG)   &=  &c(fg) \propto c(fg)E \\
        &=  &c(fg)\propto 0 \\
        &=  &c(f)c(g)\propto 0\\
        &=  &c(F)c(G).
\end{array}\)\\
Therefore $F$ is  Gaussian, as desired.

\3 Sufficiency is clear since $K\propto~K$ is a principal ring. Next assume $R$ is an arithmetical ring. We claim that $A$ is a field. Deny and let $a \not= 0\in M$ and $e\not= 0\in E$. Therefore the ideal $I :=R(a,0) + R(0,e)$ is principal in $R$ (since $R$ is local). So $I=R(a',e')$ for some $(a',e') \in R$. Clearly, $(a,0)\in I$ forces $a'$ to be nonzero and belong to $M$. Further, $(0,e)\in I$ yields $ba'=0$ and $e=be'$ for some $b\in A$. Necessarily, $b \in M$ since $a'\not=0$. It follows that $e=be'=0$, the desired contradiction. Now, let $e,e'$ be two nonzero vectors in $E$. Then $I =R(0,e)+R(0,e')$ is a principal ideal of $R$. Similar arguments used in the proof of Theorem~\ref{sec:2.1}(2) yields $e=ke'$ for some $k\in K$. So that $\dim_{K}E=1$.

\4 Let $J :=0 \propto E$ and let $\{(f_{i})\}_{i \in I}$ be a basis of the $(A/M)$-vector space $E$. Consider the exact sequence of $R$-modules:
\begin{equation}\label{eq3.1}0 \rightarrow Ker(u) \rightarrow  R^{(I)}  \buildrel u \over\rightarrow  J \rightarrow 0\end{equation}
where $u((a_{i},e_{i})_{i \in I})=(0,\sum_{i \in I}a_{i}f_{i})$. Hence, $Ker(u) =(M \propto E)^{(I)}$. Here, too, we identify $R^{(I)}$ with $A^{(I)}\propto E^{(I)}$ as $R$-modules. We claim that $J$ is not flat. Otherwise, by \cite[Theorem 3.55]{Ro}, we obtain
$J^{(I)}=(M \propto E)^{(I)}\cap JR^{(I)}=J(M \propto E)^{(I)}=0$,
absurd. By \cite[Theorem 2.4]{Ro}, $\w.dim(R)\gvertneqq 1$. Next assume that $M$ admits a minimal generating set. Then one can easily check that $M\propto E$ admits a minimal generating set too. Let $(b_{i},g_{i})_{i \in L}$ denote a minimal generating set of $M \propto E$ and consider the exact sequence of $R$-modules:
$$0 \rightarrow Ker(v) \rightarrow  R^{(L)} \buildrel v\over\rightarrow  M \propto E  \rightarrow 0$$
where $v((a_{i},e_{i})_{i \in L})=\sum_{i \in L}(a_{i},e_{i})(b_{i},g_{i})$. The minimality assumption
yields (see the proof of \cite[Lemma 4.43]{Ro}) $$Ker(v) \subseteq (M \propto E)^{(L)}.$$ It follows that $\Ker(v) =V \propto E^{(L)}=(V \propto 0) \oplus J^{(L)}$, where $$V :=\{(a_{i})_{i \in L} \in M^{(L)}\ | \sum_{i \in L}a_{i}b_{i} =0\}.$$ We obtain
\begin{equation}\label{eq3.2} \fd\big((V \propto 0) \oplus J^{(L)}\big) \leq \fd(M \propto E).\end{equation}
On the other hand, from the exact sequence in (\ref{eq3.1}) we get
\begin{equation}\label{eq3.3} \fd(M \propto E)=\fd(M \propto E)^{(I)}\leq \fd(J)-1.\end{equation}
A combination of (\ref{eq3.2}) and (\ref{eq3.3}) yields $\fd(J)\leq \fd(J)-1$. Consequently, the flat dimension of $J$ (and a fortiori the weak dimension of $R$) has to be infinite, completing the proof of the theorem.
\end{proof}

 Theorem~\ref{sec:3.1} generates new and original examples of rings with zerodivisors subject to Pr\"ufer conditions as shown below.

\begin{example}\label{sec:3.2}\rm
 Let $(V,M)$ be a non-trivial valuation domain. Then $R:=V\propto~\frac{V}{M}$ is a non-arithmetical Gaussian total ring of quotients.
\end{example}

\begin{example}\label{sec:3.3}\rm
 Let $K$ be a field and $E$ a $K$-vector space with $\dim_{K}E\geq2$. Then $R:=K\propto~E$ is a non-arithmetical Gaussian total ring of quotients.
\end{example}

\begin{example}\label{sec:3.4}\rm
 Let $(A,M)$ be a non-valuation local domain. Then $R:=A\propto~\frac{A}{M}$ is a non-Gaussian total ring of quotients.
\end{example}

Recently, Bazzoni and Glaz proved that a Gaussian ring, with a maximal ideal $M$ such that the nilradical of $R_{M}$ is non-null and nilpotent, has infinite weak dimension \cite[Theorem 6.4]{BG2}. The next example widens the scope of validity of Bazzoni-Glaz conjecture as well as illustrates the setting of this result beyond coherent Gaussian rings.

\begin{example}\label{sec:3.5}\rm
Let $\R$ denote the field of real numbers and $x$ an indeterminate over $\R$. Then $R:=\R \propto~\R[x]$ satisfies the following statements:\\
\1 $R$ is a Gaussian ring,\\
\2 $R$ is not an arithmetical ring,\\
\3 $R$ is not a coherent ring,\\
\4 $R$ is local with nonzero nilpotent maximal ideal,\\
\5 $\w.dim(R)=\infty$.
\end{example}

\begin{proof}
Assertions \1 and \2 hold by direct application of Theorem~\ref{sec:3.1}. Assertion \3 is handled by \cite[Theorem 2.6(2)]{KM}. Clearly, \4 holds since the maximal ideal of $R$ is $M:=0 \propto~\R[x]$ (by \cite[Theorem 25.1(3)]{H}) with $M^{2}=0$. Finally, \5 is satisfied by Theorem~\ref{sec:3.1}(4), \cite[Proposition 6.3]{BG2}, or \cite[Theorem 6.4]{BG2}.
\end{proof}

\section{Kaplansky-Tsang-Glaz-Vasconcelos conjecture}\label{sec:4}


Let $R$ be a ring and $Q(R)$ its total ring of quotients. An ideal $I$ of $R$ is said to be invertible if $II^{-1}=R$, where $I^{-1}:=\{x\in Q(R)\ |\
xI\subseteq R\}$. A nonzero ideal is invertible if and only if it is regular, finitely generated, and locally principal. Particularly, for finitely
generated ideals of domains, invertibility coincides with the locally principal condition. A polynomial $f$ over $R$ is said to be Gaussian if
$c(fg)=c(f)c(g)$ holds for any polynomial $g$ over $R$.

A problem initially associated with Kaplansky and his student Tsang \cite{AK,BG,GV,Lu,T} and
also termed as Tsang-Glaz-Vasconcelos conjecture in \cite{HH} sustained that ``every nonzero Gaussian polynomial over a domain has an invertible (or,
equivalently, locally principal) content ideal." It is well-known that a polynomial over any ring is Gaussian if its content ideal is locally
principal. The converse is precisely the object of Kaplansky-Tsang-Glaz-Vasconcelos conjecture extended to those rings where ``every Gaussian
polynomial has locally principal content ideal."

Notice for convenience that the conjecture has a local character since the Gaussian condition is a
local property (i.e., a polynomial is Gaussian over a ring $R$ if and only if its image is Gaussian over $R_{M}$ for each maximal ideal $M$ of $R$).
It is this very fact that has enabled a natural extension of the conjecture from domains to rings (recall, for instance, that a Von
Neumann regular ring is locally a field).

Significant progress has been made on this conjecture. Glaz and Vasconcelos proved it for normal Noetherian domains \cite{GV}. Then Heinzer and
Huneke established its veracity over locally approximately Gorenstein rings (see definition below) and over locally Noetherian domains \cite[Theorem
1.5 \& Corollary 3.4]{HH}. Recently, Loper and Roitman settled the conjecture for (locally) domains \cite[Theorem 4]{LR}, and then Lucas extended
their result to arbitrary rings by restricting to polynomials with regular content \cite[Theorem 6]{Lu}. Obviously, the conjecture is true in arithmetical rings. Moreover, trivial ring extensions offer the possibility to widen the scope of its validity to a large family of rings distinct from the above contexts. This gives birth to a new class of
rings containing strictly the three classes of arithmetical rings, of locally domains, and of locally approximately Gorenstein rings (see Figure~\ref{fig1}). We term the new concept as follows:

\begin{definition}\label{sec:4.1}\rm
 A ring $R$ is pseudo-arithmetical if every Gaussian polynomial over $R$ has locally principal content ideal.
\end{definition}

We first prove a transfer result (Theorem~\ref{sec:4.2}) on trivial ring extensions. Then Conjecture~\ref{sec:4.5} will equate the
pseudo-arithmetical notion with the local irreducibility of the zero ideal. If true, this conjecture would offer an optimal solution to the Kaplansky-Tsang-Glaz-Vasconcelos conjecture that recovers all previous results.

\begin{thm}\label{sec:4.2}
\1 Let $R:=A \propto~K$ be the trivial ring extension of a domain $A$ by its quotient field $K$. Then $R$ is a pseudo-arithmetical ring.\\
\2 Let $A\subseteq B$ be an extension of rings and $R:=A \propto~B$. If $R$ is a pseudo-arithmetical ring, then so is $A$.
\end{thm}

\begin{proof}
\1 Let $F:=\sum (a_{i},k_{i})x^{i}$ be a
nonzero Gaussian polynomial in $R[x]$. Assume $a_{i} \not= 0$ for some $i$. Then $(a_{i},k_{i})$ is regular in $R$ and so is $c(F)$ in $R$. Hence the
Gaussian property forces $F$ to be regular in $R[x]$. Then $c(F)$ is locally principal by \cite[Theorem 6]{Lu}. Next assume $a_{i} =0$ for each $i$.
Let $a$ be a nonzero element of $A$ such that $ak_{i}\in A$ for each $i$ and set $F':=(a,0)F=\sum (0,ak_{i})x^{i}$ in $R[x]$. We claim that $f'
:=\sum ak_{i}x^{i}$ is a (nonzero) Gaussian polynomial of $A[x]$. Indeed, consider $g =\sum a'_{i}x^{i} \in A[x]$ and set $G :=\sum (a'_{i},0)x^{i}$
in $R[x]$. Then $0 \propto c(f'g) =c(F'G) =c(F')c(G)$. Moreover, $c(F') =\sum R(0,ak_{i})=0 \propto c(f')$ and $c(G)=c(g)\propto K$ (see proof of
Lemma~\ref{sec:2.2}). It follows that $0 \propto c(f'g) =0 \propto c(f')c(g)$ and hence $c(f'g)=c(f')c(g)$. Whence $c(f')$ is locally principal since
$A$ is a domain \cite{LR}. Let $P:=p\propto K\in \Max(R)$ for some maximal ideal $p$ of $A$ and set $S:=(A\setminus p) \times 0\subseteq R\setminus P$. Since $c(f')A_{p} =a'A_{p}$ for some $a' \in A$, we get\\
\(\begin{array}{lcl}
(a,0)c(F)R_{P}  &=   &c(F')R_{P} \\
                &=  &\big(0\propto c(f')\big)R_{P} \\
                &=  &\big(S^{-1}(0\propto c(f'))\big)R_{P} \\
                &=  &\big(0 \propto c(f')A_{p}\big)R_{P}\\
                &=  &\big(0 \propto a'A_{p}\big)R_{P}\\
                &=  &(0,a')R_{P} \\
                &=  &(a,0)(0,\frac{a'}{a})R_{P}.
\end{array}\)\\
Consequently, $c(F)R_{P}=(0,\frac{a'}{a})R_{P}$ since $(a,0)$ is regular in $R$. Thus $c(F)$
is locally principal and therefore $R$ is a pseudo-arithmetical ring.

\2 Let $f =\sum a_{i}x^{i}$  be a Gaussian polynomial over $A$ and set $F:=\sum (0,a_{i})x^{i}$. Let $G=\sum (a'_{i},b_{i})x^{i} \in R[x]$ and set
$g:=\sum a'_{i}x^{i}$ in $A[x]$. Since $f$ is Gaussian, we have $c(F)c(G)=(0\propto c(f))c(G)=0 \propto c(f)c(g)=0 \propto c(fg)$. On the other hand,
one can see that $c(FG) =0 \propto c(fg)$. Therefore, $c(FG) =c(F)c(G)$, hence $F$ is a Gaussian polynomial over $R$. So $c(F)=0 \propto I$ is a
locally principal ideal of $R$ where $I:=c(f)$. Now ape the proof of the arithmetical statement in Lemma~\ref{sec:2.2}, to get that $I$ is locally
principal, as desired.
\end{proof}

Obviously, a ring is arithmetical if and only if it is Gaussian and pseudo-arithmetical. In this context, note that Examples \ref{sec:2.5} \& \ref{sec:3.3} illustrate the failure of Theorem~\ref{sec:4.2}\1 for trivial ring extensions $R:=A \propto~E$ with $E\not=\qf(A)$.

\begin{example}\label{sec:4.4}\rm
Let $(A,M)$ be a local ring which is not a field and $E$ a nonzero vector space over $\frac{A}{M}$. Then $R:=A\propto~E$ is a Pr\"ufer ring which is not pseudo-arithmetical. Indeed, Theorem~\ref{sec:3.1} ensures that $R$ is a non-arithmetical total ring of quotients (hence Pr\"ufer). We claim that the polynomial $f:=(a,0)+(0,e)x$, where $a\not=0\in M$ and $e\not=0\in E$, is Gaussian but $c(f)$ is not principal in $R$. To see this, let $g\in R[x]$. If $g\notin(M\times~E)[x]$, then Gauss lemma ensures that $c(fg)=c(f)c(g)$ since $R$ is local with ideal maximal $M\propto~E$. Assume $g\in (M\times~E)[x]$. Then $ME=0$ yields $$c(f)c(g)=(a,0)c(g)=c((a,0)g)=c(fg).$$
 Now ape the proof of Theorem~\ref{sec:3.1}(3), to obtain that $c(f)$ is not principal, and therefore $R$ is not pseudo-arithmetical.
\end{example}

\begin{remark}\label{sec:4.4.1}\rm
\1 Now pick any non-Pr\"ufer domain $A$ with $K :=\qf(A)$ and consider the trivial extension $R:=A \propto~K$. Then by Corollary~\ref{sec:2.3}, $R$ is
not a Pr\"ufer ring (a fortiori, $R$ is not arithmetical). Moreover, there are plenty of non-regular Gaussian polynomials over $R$, e.g., $f:=\sum
(0,k_{i})x^{i}$. However, Theorem~\ref{sec:4.2} ensures that every Gaussian polynomial over $R$ has locally principal content ideal (i.e., $R$ is pseudo-arithmetical).

\2 Next we examine the Noetherian case. From \cite{HH}, a local ring $(R,M)$ is said to be approximately Gorenstein if $R$ is Noetherian and for every
integer $n>0$ there is an ideal $I\subseteq M^{n}$ such that $R/I$ is Gorenstein (e.g., any local Noetherian ring $(R,M)$ with the $M$-adic
completion $\hat{R}$ reduced). Heinzer and Huneke proved that every locally approximately Gorenstein ring is pseudo-arithmetical \cite[Theorem
1.5]{HH}. This result combined with \cite[Remark 1.6]{HH} asserts that Noetherianity has no direct effect on the pseudo-arithmetical notion even in
low dimension, in the sense that non-Gorenstein Artinian local rings are not pseudo-arithmetical. Finally, notice that the above example $R:=A
\propto~K$ is not Noetherian since it is not coherent by \cite[Theorem 2.8]{KM}.

\3 From \cite{G2}, a ring $R$ is called a PF ring if all principal ideals of $R$ are flat, or, equivalently, if $R$ is locally a domain \cite[Theorem
4.2.2(3)]{G1}. A ring $R$ is called a PP ring or a weak Baer ring if all principal ideals of $R$ are projective. In the class of Gaussian rings, the
PP and PF properties coincide, respectively, with the notions of semihereditary ring and ring with weak dimension at most 1. Clearly, note that the above example $R:=A\propto~K$ is not locally a domain.
\end{remark}

In view of Example~\ref{sec:4.4} and Remark~\ref{sec:4.4.1}, Figure~\ref{fig1} summarizes the relations between all these classes of rings where the implications are irreversible in general.

\begin{center}
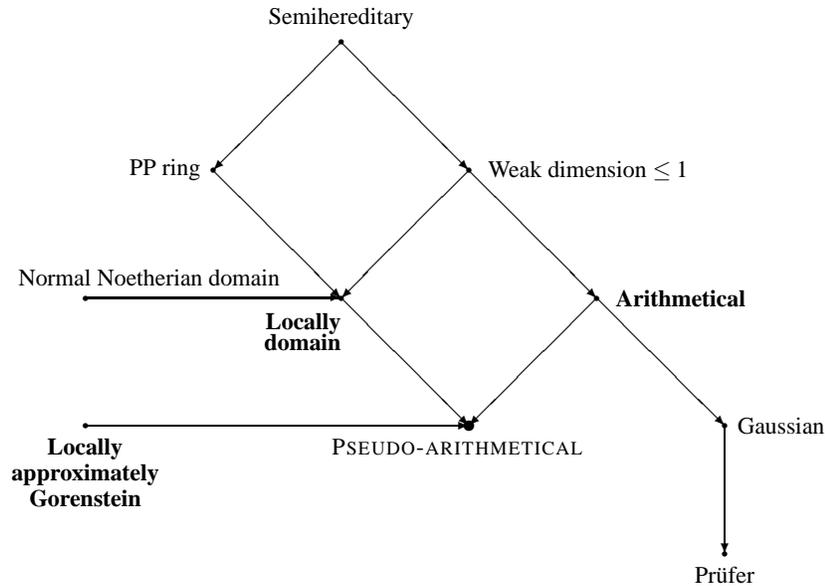
\begin{figure}[h]
\[\setlength{\unitlength}{.85mm}
\begin{picture}(100,85)(-35,-70)
\put(0,10){\vector(1,-1){20}} \put(0,10){\vector(-1,-1){20}} \put(-20,-10){\vector(1,-1){20}} \put(20,-10){\vector(1,-1){20}}
\put(20,-10){\vector(-1,-1){20}} \put(40,-30){\vector(1,-1){20}} \put(40,-30){\vector(-1,-1){20}} \put(0,-30){\vector(1,-1){20}}
\put(60,-50){\vector(0,-1){20}} \put(-40,-30){\vector(1,0){40}} \put(-40,-50){\vector(1,0){60}}
\put(0,10){\circle*{1}}\put(0,12){\makebox(0,0)[b]{\small Semihereditary}}

\put(-20,-10){\circle*{1}}\put(-22,-10){\makebox(0,0)[r]{\small PP ring}}

\put(-40,-30){\circle*{1}}\put(-30,-28){\makebox(0,0)[b]{\small Normal Noetherian domain}} \put(-40,-50){\circle*{1}} \put(-40,-52){\makebox(0,0)[t]{\small\bf
Locally}} \put(-40,-56){\makebox(0,0)[t]{\small\bf approximately}}\put(-40,-60){\makebox(0,0)[t]{\small\bf Gorenstein}}
\put(20,-10){\circle*{1}}\put(23,-10){\makebox(0,0)[l]{\small Weak dimension $\leq 1$}}

\put(0,-30){\circle*{1}} \put(0,-34){\makebox(0,0)[r]{\small\bf Locally}} \put(0,-37){\makebox(0,0)[r]{\small\bf domain}}

\put(40,-30){\circle*{1}} \put(43,-30){\makebox(0,0)[l]{\small\bf Arithmetical}} \put(20,-50){\circle*{1.5}} \put(18,-52){\makebox(0,0)[t]{\small\sc
Pseudo-arithmetical}} \put(60,-50){\circle*{1}} \put(62,-50){\makebox(0,0)[l]{\small Gaussian}} \put(60,-70){\circle*{1}}
\put(60,-72){\makebox(0,0)[t]{\small Pr\"ufer}}
\end{picture}\]
\caption{ Pseudo-arithmetical rings in perspective}\label{fig1}
\end{figure}
\end{center}

From the above discussion, it turns out that the pseudo-arithmetical notion must have a characterization that accommodates the three disparate
classes of arithmetical rings, of locally domains, and of locally approximately Gorenstein rings (see Figure~\ref{fig1}). This new characterization will offer a ``happy end"
to the Kaplansky-Tsang-Glaz-Vasconcelos conjecture. In this vein, we conjecture the following:

\begin{conjecture}\label{sec:4.5}
A ring $R$ is pseudo-arithmetical if and only if the zero ideal is locally irreducible.
\end{conjecture}

\begin{remark} \label{sec:4.6}\rm

\1  Fuchs, Heinzer and Olberding have recently studied irreducibility in commutative rings
\cite{FHO1,FHO2} and noticed that ``it is readily seen that a ring $R$ is an arithmetical ring if and only if for each proper ideal
$I$ of $R$, $I_{M}$ is an irreducible ideal of $R_{M}$ for every maximal ideal $M$ of $R$ containing $I$ \cite{FHO1}."

\2 Assume that Conjecture~\ref{sec:4.5} is true. If $R$ is locally a domain or locally approximately Gorenstein, then a polynomial over $R$ is Gaussian if and only if its content is locally principal {\cite[Theorem 4]{LR} \& \cite[Theorem 1.5]{HH}}. In particular, a nonzero polynomial over an integral domain is Gaussian if and only if its content is invertible.
Indeed the locally domain statement follows from the obvious fact that the zero ideal in a domain is irreducible. Next assume $R$ is locally approximately Gorenstein. Recall that a  Gaussian polynomial $f:=\sum a_{i}x^{i}$ over a ring $R$ forces its image
$\overline{f}:=\sum \overline{a_{i}} x^{i}$ to be Gaussian over $R/I$, for every ideal $I$ of $R$. Using this fact and the fact that the Gaussian
condition is a local property, in combination with the definition of a locally approximately Gorenstein ring, Heinzer and Huneke showed that the
proof reduces to the case where $R$ is a zero-dimensional local Gorenstein ring (see the beginning of the proof of \cite[Theorem 1.5]{HH}). But in
this setting the zero ideal is irreducible, as desired.
\end{remark}



\begin{thebibliography}{99}\par
\bibitem{AK}
D. D. Anderson and B. J. Kang, Content formulas for polynomials and power series and complete integral closure, J. Algebra  181 (1996) 82--94.\par
\bibitem{BG}    S. Bazzoni and S. Glaz, Pr\"ufer rings, Multiplicative Ideal Theory in Commutative Algebra, pp. 263--277, Springer, 2006.\par
\bibitem{BG2}   S. Bazzoni and S. Glaz, Gaussian properties of total rings of quotients, J. Algebra  310 (2007) 180--193.\par
\bibitem{BS}    H. S. Butts and W. Smith, Pr\"ufer rings, Math. Z.  95 (1967) 196--211.\par
\bibitem{CE}    H. Cartan and S. Eilenberg, Homological Algebra, Princeton University Press, 1956.\par
\bibitem{Fu}    L. Fuchs, Uber die Ideale arithmetischer Ringe, Comment. Math. Helv.  23 (1949) 334--341.\par
\bibitem{FHO1}  L. Fuchs, W. Heinzer and B. Olberding, Commutative ideal theory without finiteness conditions: Primal ideals, Trans. Amer. Math. Soc. 357 (2005) 2771-2798.\par
\bibitem{FHO2}  L. Fuchs, W. Heinzer and B. Olberding, Commutative ideal theory without finiteness conditions: Completely irreducible ideals, Trans. Amer. Math. Soc.  358 (2006) 3113-3131.\par
\bibitem{Gi}    R. Gilmer, Multiplicative Ideal Theory, Marcel Dekker, New York, 1972.\par
\bibitem{G1}    S. Glaz, Commutative Coherent Rings, Lecture Notes in Mathematics, 1371, Springer-Verlag, Berlin, 1989.\par
\bibitem{G12}   S. Glaz, Finite conductor rings, Proc. Amer. Math. Soc. 129 (2000) 2833--2843.\par
\bibitem{G2}    S. Glaz, The weak dimension of Gaussian rings, Proc. Amer. Math. Soc.  133 (9) (2005) 2507--2513.\par
\bibitem{G3}    S. Glaz, Pr\"ufer conditions in rings with zero-divisors, CRC Press Series of Lectures in Pure Appl. Math.  241 (2005) 272--282.\par
\bibitem{GV}    S. Glaz and W. Vasconcelos, The content of Gaussian polynomials, J. Algebra  202 (1998) 1--9.\par
\bibitem{Gr}    M. Griffin, Pr\"ufer rings with zero-divisors, J. Reine Angew Math.  239/240 (1969) 55--67.\par
\bibitem{HH}    W. Heinzer and C. Huneke, Gaussian polynomials and content ideals, Proc. Amer. Math. Soc.  125 (1997) 739--745.\par
\bibitem{H}     J. A. Huckaba, Commutative Rings with Zero-Divisors, Marcel Dekker, New York, 1988.\par
\bibitem{J}     C. U. Jensen, Arithmetical rings, Acta Math. Hungr.  17 (1966) 115--123.\par
\bibitem{KM}    S. Kabbaj and N. Mahdou, Trivial extensions defined by coherent-like conditions, Comm. Algebra  32 (10) (2004) 3937--3953.\par
\bibitem{K}     W. Krull, Beitrage zur arithmetik kommutativer integritatsbereiche, Math. Z.  41 (1936) 545--577.\par
\bibitem{LR}    K. A. Loper and M. Roitman, The content of a Gaussian polynomial is invertible, Proc. Amer. Math. Soc.  133 (5) (2004) 1267--1271.\par
\bibitem{Lu}    T. G. Lucas, Gaussian polynomials and invertibility, Proc. Amer. Math. Soc.  133 (7) (2005) 1881--1886.\par
\bibitem{Os}    B. Osofsky, Global dimension of commutative rings with linearly ordered ideals, J. London Math. Soc. 44 (1969) 183--185.\par
\bibitem{P}     H. Pr\"ufer, Untersuchungen uber teilbarkeitseigenschaften in korpern, J. Reine Angew. Math.  168 (1932) 1--36.\par
\bibitem{Ro}    J. J. Rotman, An Introduction to Homological Algebra, Academic Press, New York, 1979.\par
\bibitem{T}     H. Tsang, Gauss's Lemma, Ph.D. thesis, University of Chicago, Chicago, 1965.\par
\end{thebibliography}
\end{document}